\documentclass[11pt]{amsart}
\usepackage{enumitem, amsmath, amsthm, amsfonts, amssymb, xy,  mathrsfs, graphicx, lscape}
\usepackage[usenames, dvipsnames]{color}
\usepackage[margin=1in]{geometry}
\usepackage{tikz}
\usepackage{pgfplots}
\usepackage{mathtools}
\usepackage{comment}
\usetikzlibrary{matrix,arrows, calc}
\pgfplotsset{compat=1.17}
\usepackage{tikz-cd}
\usepackage{xcolor}
\usepackage{stackengine}
\usepackage{graphicx}

\theoremstyle{plain}

\newtheorem{thm}{Theorem}[section]
\newtheorem*{thm*}{Theorem}
\newtheorem{cor}[thm]{Corollary}
\newtheorem{lem}[thm]{Lemma}
\newtheorem{claim}[thm]{Claim}
\newtheorem{prop}[thm]{Proposition}

\newtheorem{conj}[thm]{Conjecture}
\theoremstyle{definition}
\newtheorem{defn}[thm]{Definition}

\theoremstyle{remark}

\newtheorem{rmk}[equation]{Remark}
\newenvironment{remark}[1][]{\begin{rmk}[#1] }{\popQED \end{rmk}}
\newtheorem{eg}[equation]{Example}

\newtheorem{ques}[equation]{Question}

\usepackage{bbm}

\usepackage[colorlinks=true, pdfstartview=FitV, linkcolor=black, citecolor=black, urlcolor=black]{hyperref}

\newcommand{\Z}{\mathbb{Z}}
\newcommand{\C}{\mathbb{C}}
\newcommand{\N}{\mathbb{N}}

\renewcommand{\phi}{\varphi}
\renewcommand{\emptyset}{\varnothing}

\makeatletter
\def\Ddots{\mathinner{\mkern1mu\raise\p@
\vbox{\kern7\p@\hbox{.}}\mkern2mu
\raise4\p@\hbox{.}\mkern2mu\raise7\p@\hbox{.}\mkern1mu}}
\makeatother

\numberwithin{equation}{section}

\usepackage[utf8]{inputenc}

\usepackage[english]{babel}
\usepackage{csquotes}
\usepackage{mathrsfs}
\usepackage{comment}
\usepackage{tikz-cd}

\usepackage{xcolor}

\usepackage{hyperref}

\hypersetup{colorlinks,
    linkcolor={red!50!black},
    citecolor={blue!80!black},
    urlcolor={blue!80!black}}
\usepackage{float}

\numberwithin{equation}{section}

\renewcommand{\phi}{\varphi}

\renewcommand{\kappa }{\varkappa}

\newcommand{\lab}{{\bf Lab}}

\newcommand{\cL}{\mathcal{L}}
\newcommand{\M}{\mathcal M}

\newcommand{\ra}{\rightarrow}

\newcommand{\A}{\mathcal A}

\newcommand{\cN}{\mathcal{N}}
\newcommand{\cM}{\mathcal{M}}

\theoremstyle{plain} 
\usepackage{hyperref}
\newtheorem*{genericthm*}{\thistheoremname}
\newcommand{\thistheoremname}{???}
\newcounter{genericthm}

\newenvironment{namedthm*}[1]
  {\renewcommand{\thistheoremname}{#1}%
   \refstepcounter{genericthm}%
   \begin{genericthm*}}
  {\end{genericthm*}}

\newtheorem{main}{Theorem}

\newtheorem{thma}[main]{Theorem}

\newtheorem{mcor}[main]{Corollary}

\title{Amenable absorption in von Neumann algebras of hyperbolic groups}

\author{Juan Felipe Ariza Mej\'{i}a}
\address{Department of Mathematics, The University of Iowa, 14 MacLean Hall,  Iowa City, IA 52242, USA}
\email{juanfelipe-arizamejia@uiowa.edu}

\author{Ionu\c{t} Chifan}
\address{Department of Mathematics, The University of Iowa, 14 MacLean Hall, Iowa City, IA 52242, USA}
\email{ionut-chifan@uiowa.edu}

\author{Adriana Fern\'{a}ndez Quero}
\address{Department of Mathematics, KU Leuven, Celestijnenlaan 200B, 3001 Leuven, Belgium}
\email{adriana.fernandeziquero@kuleuven.be}

\author{Adrian Ioana}
\address{Department of Mathematics, University of California San Diego, 9500 Gilman Drive, La Jolla, CA 92093, USA}
\email{aioana@ucsd.edu}

\begin{document}

\begin{abstract}
We prove that the von Neumann algebra $\cL(G)$ associated with any hyperbolic group $G$ satisfies the following \emph{amenable absorption property}: for any infinite maximal amenable subgroup $H \leqslant G$ and any amenable von Neumann subalgebra $\mathcal{Q} \subset \cL(G)$ with diffuse intersection with $\cL(H)$, one must have $\mathcal{Q} \subset \cL(H)$. This strengthens a result of Boutonnet and Carderi 
\cite{BC2}. 

We also establish similar amenable absorption results for the broader class of acylindrically hyperbolic groups, including relatively hyperbolic groups, mapping class groups, and limit groups.
\end{abstract}

\maketitle

\section{Introduction}

Amenable von Neumann algebras represent the most fundamental and best understood class of von Neumann algebras. This class has been studied extensively since the pioneering work of Murray and von Neumann \cite{MvN43}, who established the uniqueness of the hyperfinite II$_1$ factor. Connes' celebrated work \cite{Co76} characterized hyperfinite von Neumann algebras as precisely those that are amenable. As a result, every subalgebra of an amenable tracial von Neumann algebra is hyperfinite. In contrast, no analogous structural description exists for non-amenable von Neumann algebras.


To address this shortcoming, given a non-amenable II$_1$ factor, one studies the structure of its amenable subalgebras. In the 1960s, Kadison asked whether any maximal amenable subalgebra of a II$_1$ factor is necessarily a factor. Popa answered this question in the negative in his breakthrough work \cite{Po83}. To do so, he introduced a powerful method based on asymptotic orthogonality within the ultraproduct framework to prove maximal amenability results for tracial von Neumann algebras. Remarkably, Popa used this technique to show that the generator MASA $\cL(\langle a\rangle)$ in any free group factor $\cL(\mathbb{F}_n)$ is maximal amenable. This asymptotic orthogonality property, along with its subsequent refinements, has since become central to the study of maximal amenability; see, for instance, \cite{She,Fan,CFRW,Gao,Ho14,Bro,KEP24}.


The study of maximal amenability, and, more generally, amenable absorption phenomena, has witnessed significant progress over the past decade.  By combining techniques from geometric group theory with the analysis of central sequences, Boutonnet and Carderi \cite{BC2} showed that if a group $G$ is relatively hyperbolic with respect to a subgroup $H$, then the inclusion $\cL(H)\subset \cL(G)$ is maximal Gamma in the sense of Murray--von Neumann. Combined with Ozawa's solidity theorem \cite{Oz04}, this implies that whenever $G$ is hyperbolic and $H\subset G$ is a maximal amenable subgroup, the von Neumann algebra $\cL(H)$ is maximal amenable inside $\cL(G)$, thereby answering a question posed by Houdayer \cite[Problème 3.13]{Ho13}. In subsequent work \cite{BC1}, Boutonnet and Carderi introduced a new approach to maximal amenability based on the analysis of invariant measures and extensions of traces. Centered around the notion of a singular subgroup, 
their methods provide general criteria for maximal amenability. In particular, they recover their earlier results and yield new examples arising from amalgamated free product groups and related constructions.

Around the same time, Houdayer \cite[Theorem A]{Ho14} established a general Gamma stability theorem for free product von Neumann algebras by extending Popa's asymptotic orthogonality techniques to the free product setting. Shortly thereafter, Houdayer and Ueda \cite[Theorem A]{HU15} completely resolved the problem of maximal amenability for arbitrary free product von Neumann algebras, proving that if $\cM=\cM_1 * \cM_2$ and $\cM_1$ is amenable, then  $\cM_1\subset \cM$ is maximal amenable.

Building on techniques introduced in \cite{BC1}, Ozawa \cite[Theorem]{Oz15} provided a new proof of Houdayer's amenable absorption theorem \cite[Theorem 4.1]{Ho14} for tracial free product von Neumann algebras. Later, Boutonnet and Houdayer \cite[Main Theorem]{BH16} established a broad amenable absorption theorem for arbitrary amalgamated free products of von Neumann algebras through a detailed analysis of non-normal conditional expectations. Their result simultaneously generalized \cite[Corollary B]{BC1}, \cite[Corollary B]{HU15}, and \cite[Theorem 3.3]{leary}. 

More recently, Hayes, Jekel, Nelson, and Sinclair \cite[Theorem A]{HJNS} obtained a new proof of amenable absorption in the free product setting via free entropy techniques, while Houdayer and Ioana \cite[Theorem E]{HI23} found an alternative proof of \cite[Main Theorem]{BH16} in the tracial case.

In connection with these results,  it was shown in \cite[Theorem 5.8]{AMAKCK} that for a group $G$ with a hyperbolically embedded subgroup $H$, any abelian subalgebra $\mathcal{A} \subset \mathcal{L}(G)$ whose normalizer intersects $\mathcal{L}(H)$ diffusely must be absorbed into $\mathcal{L}(H)$. This property, which mirrors an older principle discovered 
in \cite{IPP05} in the context of amalgamated free products, is used in an essential way to derive the main computation of the one-sided fundamental semigroup of a class of II$_1$ factors arising from property (T) relatively hyperbolic groups.

\subsection{Main Results}
In this paper, we obtain a new amenable absorption result for von Neumann algebras of groups which admit hyperbolically embedded subgroups \cite[Definition 2.1]{DGO}. Note that a countable group $G$ contains a proper infinite hyperbolically embedded subgroup if and if it is acylindrically hyperbolic \cite[Theorem 1.2]{Osi16}.

\begin{thma}\label{amenableabs1}
Let $G$ be a group with an infinite hyperbolically embedded subgroup $H$. Then any amenable von Neumann subalgebra $\mathcal{Q} \subset \mathcal{L}(G)$ such that $\mathcal{Q} \cap \mathcal{L}(H)$ is diffuse satisfies $\mathcal{Q} \subset \mathcal{L}(H)$.
\end{thma}

Examples of groups $G$ with an infinite subgroup $H$ that is hyperbolically embedded (in symbols, $H \hookrightarrow_h G$) include the following:

\begin{enumerate}
    \item  All non-elementary hyperbolic groups $G$ and every subgroup $H=E(a)\hookrightarrow_h G$, where $a\in G$ is any element of infinite order and
\[
E(a)=\{g\in G \,:\, ga^n g^{-1}=a^{\pm n}\text{ for some } n>0\}
\]
is the maximal virtually cyclic subgroup of $G$ containing $a$ \cite{DGO}. More generally, we have $H\hookrightarrow_h G$ for any quasiconvex, almost malnormal subgroup $H<G$, by \cite[Theorem 7.11]{Bo12}, Osin's characterization of relative hyperbolicity via linear relative Dehn functions, and \cite[Definition 4.25 and Theorem 4.24]{DGO}.

  \vspace{0.5mm}
  \item All groups $G=\pi_1(M)$, where $M=\mathbb H^n/\Gamma$ is a finite-volume, noncompact real hyperbolic $n$-manifold, and any of their cusp subgroups $H\cong \mathbb Z^{n-1}$ satisfy $H\hookrightarrow_h G$. Similarly, all groups $G=\pi_1(M)$, where $M=\mathbb C\mathbb H^n/\Gamma$ is a finite-volume complex hyperbolic $n$-manifold, and any of their cusp subgroups $H$ satisfy $H\hookrightarrow_h G$, where $H$ is virtually isomorphic to the integral Heisenberg group $H_{2n-1}(\mathbb Z)$ \cite{Osi06}.  
   
 \item Mapping class group, $G = \mathcal{MCG}(\Sigma)$, where $\Sigma$ is a closed orientable surface. This group consists of isotopy classes of orientation-preserving homeomorphisms of $\Sigma$. For every pseudo-Anosov $a \in G$ (a homeomorphism that acts on a closed surface by stretching and shearing it), the associated maximal virtually cyclic subgroup $H = E(a)$ is hyperbolically embedded in $G$, \cite[Theorem 2.19]{DGO}.

    \vspace{0.5mm}

    \item Outer automorphisms of free groups, $G = \mathrm{Out}(\mathbb{F}_n)$, where $n \geq 2$. In this case, for any irreducibe with irreducible powers element $g \in G$ \cite[Section 4]{LM} the virtually cyclic group $H = E(g)$ is hyperbolically embedded in $G$, \cite[Theorem 2.20]{DGO}.


    \vspace{0.5mm}

    \item Groups acting property on a proper $\text{CAT}(0)$ space. Let $g\in G$ be a rank one isometry, then $g$ is contained in a unique maximal virtually cyclic subgroup which is hyperbolically embedded in $G$, \cite{Sisto}.

    \vspace{0.5mm}

    \item   
    All amalgamated free groups and HNN-extensions obtained in \cite[Theorem 2.25]{DGO} via standard combinations theorems: \begin{enumerate}\item For any groups $H \hookrightarrow_h A$ and $H \hookrightarrow_h B$, consider their amalgamated free product $G=A\ast_H B$ and hence we have $H\hookrightarrow_h G$. \item For any groups $\{K,H\}\hookrightarrow_h A$ and monomorphism  $\iota: K \ra H$, consider the HNN-extension $G =\langle A, t\,:\, tkt^{-1}=\iota(k) \text{ for all }k\in K \rangle $   
    and hence we have $H \hookrightarrow_h G$.
    
    \end{enumerate}
    \vspace{0.5mm}
    
    \item Directly indecomposable, centerless right-angled Artin groups $G=A(\Gamma)$ and the unique maximal virtually cyclic subgroup $H=E(g)\hookrightarrow_h G$ containing an element
$g\in A(\Gamma)$ that acts as a rank-$1$ isometry on the universal cover of the associated Salvetti complex \cite{Osi16,KatoOguni}.

\end{enumerate}

Theorem \ref{amenableabs} immediately implies that if $H$ is virtually abelian, then any type II$_1$ amenable von Neumann subalgebra of $\mathcal{L}(G)$ must have a completely atomic intersection with $\mathcal{L}(H)$. Additional consequences of Theorem \ref{amenableabs1} explore the behavior in corners by nonzero projections (Corollary \ref{absorptioncorner}) and normalizers (Corollary \ref{absorptionnormalizer}); for additional results, we refer the reader to Section \ref{applicationsam}.
 
In particular, Theorem \ref{amenableabs1} allows us to give an alternative proof of the following result of Boutonnet and Carderi \cite{BC2} in the setting of hyperbolically embedded subgroups. 

\begin{mcor}\label{maximalamencor}
Let $H < G$ be groups, where $H$ is infinite, amenable, and hyperbolically embedded in $G$. Then $\mathcal{L}(H)$ is a maximal amenable von Neumann subalgebra of $\mathcal{L}(G)$. 
\end{mcor}


To establish Theorem \ref{amenableabs1}, we follow an approach based on an idea introduced by Ozawa in \cite{Oz15}. By combining an asymptotic freeness result for peripheral subgroups from \cite[Proposition 4.14]{DGO} (see also \cite[Theorem 2.8]{CDS}) with additional operator algebraic techniques, including the spectral formula and the classical Lyapunov convexity theorem, we show that for every $x \in \cL(G)\ominus \cL(H)$, there exists a sequence of unitaries $(u_i)_i \subset \cL(H)$ supported sufficiently deep in $H$ such that the averages
\[
\frac{1}{n}\sum_{i=1}^n u_i x u_i \otimes (u_i x u_i)^{\rm op}
\]
converge to zero in norm. 
Applying the methods from \cite{Oz15} then yields the desired amenable absorption.
We note in passing that our arguments involve moment computations reminiscent of classical techniques from free probability theory.

\vspace{1mm}

Next, we note that the results from \cite[Proposition 4.14]{DGO} (see also \cite[Theorem 2.8]{CDS}), which were essential to establishing our main theorems, can also be applied to prove the following standalone asymptotic freeness result.


\begin{thma}\label{amalgfreeprodinLGw} Let $G$ be a group with a hyperbolically embedded subgroup $H$ and $\omega$ a free ultrafilter on $\N$. Then the von Neumann subalgebra of  $\cL(G)^\omega$ generated by  $\cL(G)$ and $ \cL(H)^{\omega}$ decomposes as an amalgamated free product: 
\begin{equation*}\label{L(G)L(H)warefree1}
\cL(G)\vee\cL(H)^{\omega}\cong\cL(G)\ast_{\cL(H)}\cL(H)^{\omega}.
\end{equation*}
\end{thma}

Amenable absorption is closely connected to a deep conjecture  formulated by Peterson and Thom \cite{PT11}, which states that every diffuse amenable subalgebra of a free group factor $\mathcal{L}(\mathbb{F}_r)$, $r > 1$, is contained in a unique maximal amenable von Neumann subalgebra. 
Although several special cases were established over the years \cite{HU15,Ho14,BH16}, the full conjecture was recently resolved by combining breakthroughs in free probability theory \cite{Hay22} and random matrix theory \cite{BC22,BC23}.

More generally, this conjecture is expected to hold for the von Neumann algebras $\cL(G)$ of all hyperbolic groups $G$ (see Conjecture \ref{PetersonThomhyperbolic}). To establish this, it suffices to show that if two von Neumann subalgebras $\mathcal A,\mathcal B \subset \cL(G)$ have a diffuse intersection $\mathcal A \cap\mathcal B$, then the von Neumann algebra they generate, $\mathcal A \vee \mathcal B\subset \mathcal L(G)$ is amenable. Theorem \ref{amenableabs1} provides strong evidence for this conjecture by confirming it when $\mathcal A$ is any amenable subalgebra and $\mathcal B=\cL(H)$ arises from an infinite amenable subgroup $H<G$.

\textbf{Acknowledgments.} The second author is grateful to  Denis Osin for conversations around acylindrically hyperbolic groups during the workshop BIRS 24w5174.  The third author thanks Ilijas Farah, Chris Schafhauser, Gábor Szabó, and Jesse Peterson for helpful discussions regarding conditional expectations at the 2025 M. F. Oberwolfach Workshop on C$^*$-algebras, and especially thanks Jesse Peterson for permission to include the proof.

The first author received support from the NSF Grant DMS-2154637. The second author has been supported by the NSF Grant DMS-2452247. The third author received partial support from the NSF grant DMS-2154637, the CLAS Dissertation Writing Fellowship and the FWO research project G016325N of the Research Foundation Flanders. The fourth author has been supported from the NSF Grants  DMS-2153805 and DMS-2451697.

\textbf{AI tool disclosure.} Gemini was used for English language editing, proofreading, and grammatical corrections. All mathematical content was generated solely by the human authors.


\section{Preliminaries}

\subsection{Tracial von Neumann algebras}

Let $(\cM,\tau)$ be a tracial von Neumann algebra, that is, a pair consisting of a von Neumann algebra $\cM$ and a trace (i.e., faithful normal tracial state) $\tau:\cM\to\C$. For $x\in \cM$, we denote by $\|x\|$ the {\it operator norm} of $x$ and by $\|x\|_2=\tau(|x|^2)^{\frac{1}{2}}$ the {\it $2$-norm} of $x$, where $|x|=(x^*x)^{1/2}$. We denote by $\mathscr{U}(\cM)$ the group of unitaries of $\cM$ and by $(\cM)_1=\{x\in \cM\mid \|x\|\leq 1\}$ the {\it unit ball} of $\cM$. For a von Neumann subalgebra $\cN\subset\cM$, we denote by $\mathbb{E}_{\cN}:\cM\to\cN$ the unique trace-preserving conditional expectation onto $\cN$.

For a nonprincipal ultrafilter $\omega$ on $\N$, we denote by $(\cM^{\omega},\tau^{\omega})$ the tracial ultraproduct von Neumann algebra. This is defined as the quotient $\ell^{\infty}(\N,\cM)/\mathcal{J}$ by the closed ideal $\mathcal{J}\subset\ell^{\infty}(\N,\cM)$ consisting of $x=(x_n)_n$ with $\lim_{n\to\omega}\|x_n\|_2=0$,  endowed with the trace $\tau^\omega$ given by $\tau^\omega(y)=\lim_{n\to\omega}\tau(y_n)$, for every $y=(y_n)_n\in\cM^\omega$. 
We have a natural diagonal inclusion $\iota:\cM\to\cM^{\omega}$ given by $\iota(x)=(x_n)_n$, where $x_n=x$ for all $n$. For 
simplicity, we identify $\cM$ with $\iota(\cM)$.

Let $(\cM_1,\tau_1)$ and $(\cM_2,\tau_2)$ be tracial von Neumann algebras with a common von Neumann subalgebra $B$ such that $\tau_1|_B=\tau_2|_B$. 
The amalgamated free product von Neumann algebra $\cM=\cM_1*_B\cM_2$ is the unique (up to isomorphism) tracial von Neumann algebra $(\cM,\tau)$ which is generated by $\cM_1$ and $\cM_2$ in such a way that the trace $\tau$ satisfies
  $\tau(x_1\cdots x_n)=0$, whenever $n\geq 1$, $i_1,\cdots,i_n\in\{1,2\}$, $i_1\not= i_2\not=\hdots\not=i_n$ and $x_j\in\cM_{i_j}$ satisfy $\mathbb{E}_B(x_j)=0$, for every $1\leq j\leq n$. 
We say that a product of the form $x_1\cdots x_n$, where $x_j\in \cM_{i_j}\ominus B$ and $i_1\neq i_2\neq\hdots\neq i_n$, is a {\it reduced word}.

The following lemma must be known to the experts, but for lack of reference we include a proof.

\begin{lem}\label{amalgam}
        Let $(\mathcal M,\tau)$ be a tracial von Neumann algebra, $\mathcal M_1,\mathcal M_2\subset \mathcal M$ be von Neumann algebras and $B=\mathcal M_1\cap \mathcal M_2$. Assume that $\tau(x_1y_1\cdots x_my_m)=0$, for every $m\geq 1$ and  $x_1,\dots,x_m\in \mathcal M_1,y_1\dots,y_m\in \mathcal M_2$ such that $\mathbb{E}_B(x_1)=\dots=\mathbb{E}_B(x_m)=0$ and $\mathbb{E}_B(y_1)=\dots=\mathbb{E}_B(y_m)=0$.

        Then $\mathcal M_1$ and $\mathcal M_2$ are free with amalgamation over $B$.
\end{lem}

\begin{proof}
 To prove the conclusion we have to argue that for every $n\geq 1$, $i_1,\dots, i_n\in\{1,2\}$ such that $i_1\not= i_2\not=\cdots\not=i_n$ and $z_j\in \mathcal M_{i_j}$ such that $\mathbb{E}_B(z_j)=0$, for every $1\leq j\leq n$, we have $\tau(z_1\cdots z_n)=0$.  

 We prove this claim by induction on $n\geq 1$. For $n=1$, the claim is obvious. Assume that the claim holds for some $n\geq 1$. Let $i_1,\dots, i_{n+1}\in\{1,2\}$ such that $i_1\not= i_2\not=\hdots\not=i_{n+1}$ and $z_j\in \mathcal M_{i_j}$ such that $\mathbb{E}_B(z_j)=0$, for every $1\leq j\leq n+1$. If $n+1$ is even, 
 then $i_1\not= i_{n+1}$ and since $\tau(z_1\cdots z_{n+1})=\tau(z_{n+1}z_1\cdots z_n)$, the hypothesis implies that $\tau(z_1\cdots z_{n+1})=0$. If $n+1$ is odd, we have $i_1=i_{n+1}$ and letting $z=z_{n+1}z_1\in \mathcal M_{i_1}$ we get that $$\tau(z_1\cdots z_{n+1})=\tau(zz_2\cdots z_n)=\tau((z-\text{E}_B(z))z_2\cdots z_{n})+\tau(\text{E}_B(z)z_2z_3\cdots z_n).$$
 Since by the induction hypothesis we have $\tau((z-\mathbb{E}_B(z))z_2\cdots z_{n})=\tau(\mathbb{E}_B(z)z_2z_3\cdots z_n)=0$, we conclude that $\tau(z_1\cdots z_{n+1})=0$. This finishes the proof of the inductive step and of the claim.
\end{proof}

We continue by proving the following 
lemma, which is a key step in constructing a sequence of unitaries that are mutually orthogonal with respect to the trace when tested against a fixed finite set. In the proof, we use Lyapunov's convexity theorem, which was also recently used in the resolution of a question of Kadison \cite{HTZ26}.

\begin{lem}\label{oneunitary} Let $\A$ be a diffuse abelian von Neumann algebra with a faithful normal state $\tau$. For any finite set $F \subset \A$ one can find a unitary $u\in \mathscr U(\A)$ satisfying \[\tau(ux)=0,\: \text{ for all }x\in F.\]
\end{lem}
\begin{proof} After replacing $\A$ by a separable diffuse von Neumann subalgebra of $\A$ which contains $F$, we may assume that $\A$ is separable.
Identify $\A$ with $L^\infty(X, \mu)$, where $(X,\mu)$ is a non-atomic standard probability space, such that $\tau(x)= \int_X x d \mu$ for all $x\in L^\infty(X)$. Let $\Sigma$ be the $\sigma$-algebra of measurable subsets of $X$. Write $F = \{x_1, \dots, x_n\}$. Consider the map $\Phi:\Sigma \rightarrow \mathbb C^n$ given by 
\[ \Phi(E) =\left(\int_E x_1 d \mu, \int_E x_2  d\mu , \ldots ,\int_E x_n d\mu\right).\]
Since the  measure $\mu_i(E) := \int_E x_i d \mu$ is absolutely continuous with respect to $\mu$, it is non-atomic, for every $1\leq i\leq n$. Moreover, $\Phi(\emptyset)=(0,...,0)$ and $\Phi(X)= \left ( \int_X x_1d\mu,\ldots ,\int_X x_n d\mu\right )$. By Lyapunov's convexity theorem the image of $\Phi$ is a convex (and compact) subset of $\mathbb C^n$. In particular, we can find $E\in\Sigma$ such that $\Phi(E)=\frac{1}{2}\Phi(X)$. This entails that 
\begin{equation}\label{scaling} \int_E x_i d\mu = \frac{1}{2}\int_X x_i d\mu,\:\text{ for all } i\in\{1,\ldots,n\}.\end{equation}
If $u=\chi_E -\chi_{X\setminus E}\in \mathscr U(\A)$, then \eqref{scaling} implies that for every $1\leq i\leq n$ we have
\begin{equation*}\begin{split}\tau(ux_i)&= \int_E x_i d\mu - \int_{X\setminus E} x_id\mu= 2\int_E x_i d\mu -\int_X x_i d\mu =0. \end{split}\end{equation*}
This finishes the proof of our statement.
\end{proof}

\begin{thm}\label{orthunitaries}
Let $\A$ be a diffuse abelian von Neumann algebra with a faithful normal state $\tau$. Let $F \subset \A$ be a finite set.
Then there exists a sequence of unitaries $u_n \in \A$ such that
\[\tau(u^*_nu_mx)=0,\: \text{ for all } x \in F \text{ and }n\neq m.\]
\end{thm}

\begin{proof} 

We proceed by induction on $n$. For $n=1$, take $u_1=1$. Let $n\geq 1$ and suppose we have constructed $u_1, \ldots, u_n\in \mathscr U(\A)$ such that 
\begin{equation}\tau(u_i^*u_j x)= 0,\: \text{ for all } i\neq j\in \{1,\ldots,n\} \text{ and } x\in F. \end{equation} 
To construct $u_{n+1}$, we use Lemma \ref{oneunitary} for the finite set $K=\bigcup^n_{i=1}(F\cup F^*)u_i^*\subset \A$ to find a unitary $u_{n+1}$ with $\tau(u_{n+1} y)=0$ for all $y\in K$. Hence, for all $1\leq i\leq n$ and $x\in F$, we have 
\begin{equation*}
    \tau(u_i^*u_{n+1} x)=\tau(u_{n+1} xu_i^*)=0\:\text{ and }\:\tau( x^* u_i^* u_{n+1})= \tau(u_{n+1} x^* u_i^*)=0. 
\end{equation*}
Further, we get $\tau(u^*_{n+1} u_i x)=\overline{\tau(x^* u_i^* u_{n+1})}=0$ for all $1\leq i\leq n$ and $x\in F$, finishing the proof. \end{proof}

\begin{cor}\label{orthunitariesinM} Let $(\M,\tau)$ be a diffuse tracial von Neumann algebra, $\A \subset \M$ a diffuse abelian von Neumann subalgebra and $F \subset \M$ a finite set.
Then there exists a sequence of unitaries $u_n \in \mathcal A$ such that
\[\tau(u^*_nu_mx)=0,\:\text{ for all } x \in F \text{ and }n\neq m.\]    
\end{cor}

\begin{proof} Consider the finite set $K=\{\mathbb E_\A(x) \,:\, x\in F\}\subset \A$. Applying Theorem \ref{orthunitaries} to the algebra $\A$ and $K$, we find a sequence of unitaries $u_n\subset \A$ such that $\tau(u^*_nu_m \mathbb  E_{\A}(x))=0$, for all $x\in F$ and $n\neq m$. Since $u^*_nu_m\in\A$, we get $\tau(u^*_nu_mx)=\tau(u^*_nu_m \mathbb  E_{\A}(x))=0$, for all $x\in F$ and $n\neq m$.
\end{proof}

\subsection{Acylindrically hyperbolic groups}

A group $G$ is \textit{hyperbolic} if it is generated by a finite set $X$ and its Cayley graph $\mathrm{Cay}(G,X)$ is a hyperbolic metric space. This definition is independent of the choice of a particular generating set.

Let $H$ be a subgroup of a group $G$. We say that $X\subset G$ is a \textit{relative generating set} of $G$ (with respect to $H$) if $X\cup H$ generates $G$. To a relative generating set $X\subset G$, one associates the Cayley graph $\mathrm{Cay}(G,X\sqcup H)$, where the disjoint union means that for any element $a\in X\cap H$ and any vertex $g\in G$, there are two edges in $\mathrm{Cay}(G,X\sqcup H)$ going from $g$ to $ga$: one labeled by a copy of $a$ from $X$ and the other by a copy of $a$ from $H$. We denote by $\mathrm{Cay}(H,H)$ the Cayley graph of $H$ with respect to the generating set $H$ and naturally view it as a (complete) subgraph of $\mathrm{Cay}(G,X\sqcup H)$.

\begin{defn}\label{hyperbolicembedded} A subgroup $H$ \textit{hyperbolically embeds into a group $G$ with respect to a set $X\subset G$}, denoted by $H\hookrightarrow_h(G,X)$, if the following hold:

\begin{enumerate}
    \item[(i)] $X\cup H$ generates $G$ as a group,
    \item[(ii)] $\mathrm{Cay}(G,X\sqcup H)$ is hyperbolic, and
    \item[(iii)]  for any $n\in\N$, there are only finitely many elements $h\in H$ such that $h$ can be connected to $1$ in $\mathrm{Cay}(G,X\sqcup H)$ by a path of length at most $n$ avoiding edges of $\mathrm{Cay}(H,H)$.
\end{enumerate}

We say \textit{$H$ hyperbolically embeds into $G$} if there exists a set $X\subseteq G$ such that (i), (ii) and (iii) hold. The group $G$ is said to be \textit{acylindrically hyperbolic} if it contains a proper infinite hyperbolically embedded subgroup $H$ \cite{Osi16}.
\end{defn}

Examples of hyperbolically embedded subgroups include $G\hookrightarrow_h(G,\emptyset)$ and $H\hookrightarrow_h(G,G)$ for any group $G$ and any finite subgroup $H\leqslant G$. Further, we have $H\hookrightarrow_hH\ast\Z$ but $H\not\hookrightarrow_hH\times\Z$.

The following theorem is a direct consequence of the $n$-gon inequality \cite[Proposition 4.14]{DGO} (see also \cite[Proposition 3.2]{Osi07}). It was used in \cite{CDS} to verify the ISR conjecture for all ICC acylindrically hyperbolic groups. Their proof relied on robust group-theoretic properties of the hyperbolically embedded subgroup $H\leqslant G$. 
We will use this theorem in the norm computation carried out in the main theorem (see Theorem \ref{normconv0}). Moreover, we use this result to establish an amalgamated free product decomposition inside the ultrapower $\cL(G)^{\omega}$ in Section \ref{sec:asymptoticfreefactors}.

\begin{thm}[{\cite[Theorem 2.8]{CDS}}]\label{evenlength} Let $G$ be a group with a hyperbolically embedded subgroup $H$. Then, for every finite set $K\subset G\setminus H$, there exists a finite set $L\subset H$ such that for all $m,n\geq 1$,
\begin{equation*}
    ((H\setminus L)K)^{m}\cap (K(H\setminus L))^{n}=\emptyset.
\end{equation*}
Here, $(S_1S_2)^{m}=S_1S_2S_1S_2\cdots S_1S_2$ where there are $m$ copies of $S_1S_2$ on the right-hand side.
\end{thm}

We note that the previous result can be rephrased in the following way, which will be more convenient to work with in order to derive the main result in the next section.

\begin{cor}\label{nontrivial} Let $G$ be a group with a hyperbolically embedded subgroup $H$. Then for every finite set $K\subset G\setminus H$, there exists a finite set $L\subset H$ such that the following holds: for every $\ell\geq 1$ and every $g_1,\dots,g_\ell\in K, h_1,\dots,h_\ell\in H\setminus L$, we have
$g_1h_1\cdots g_\ell h_\ell\not=1$.
\end{cor}

\begin{proof}
After replacing $K$ with $K\cup K^{-1}$, we may assume that $K$ is symmetric. By Theorem \ref{evenlength} we can find a finite set $L\subset H$ such that $((H\setminus L)K)^{m}\cap (K(H\setminus L))^{n}=\emptyset$, for every $m,n\geq 1$. Since this condition holds when we replace $L$ by any larger subset of $H$, after replacing $L$ with $L\cup L^{-1}$, we may also assume that $L$ is symmetric. Since $K$ and $L$ are symmetric, the last condition is equivalent to $1\notin  (K(H\setminus L))^{m+n}$, for every $m,n\geq 1$. Thus, the conclusion holds for every $\ell\geq 2$. The conclusion also clearly holds for $\ell=1$ since $K\subset G\setminus H$ and hence $1\notin K(H\setminus L)$.
\end{proof}

\section{Asymptotically free factors}\label{sec:asymptoticfreefactors}

In this section, of independent interest, we introduce the notion of asymptotically free factors and prove
 that $\cL(H)$ is an asymptotically free factor of $\cL(G)$, for a hyperbolically embedded subgroup $H<G$, thereby establishing Theorem \ref{amalgfreeprodinLGw}.

Let $G$ be a group. A subgroup $H\leq G$ is said to be a {\it free factor} of $G$ if $G=H*K$, for a subgroup $K\leq G$.
Let $(\mathcal M,\tau)$ be a tracial von Neumann algebra. A von Neumann subalgebra $\mathcal N\subset \mathcal M$ is said to be a {\it free factor} of $\mathcal M$ if $\mathcal M=\mathcal N*\mathcal P$, for a von Neumann subalgebra $\mathcal P\subset \mathcal M$. Here we introduce the following asymptotic version of these notions:

\begin{defn} Let $\omega$ be a nonprincipal ultrafilter on $\mathbb N$. 

\begin{enumerate}
  \item  Let $G$ be a countable group.  A subgroup $H\leq G$ is said to be an {\it asymptotically free factor} of $G$ if $G$ and $H^\omega$ are free with amalgamation over $H$ inside $G^\omega$. In other words, we have a canonical isomorphism $G\vee H^\omega\cong G*_H{H^\omega}$.

\item Let $(\mathcal M,\tau)$ be a separable tracial von Neumann algebra. A von Neumann subalgebra $\mathcal N\subset \mathcal M$ is said to be an {\it asymptotically free factor} of $\mathcal M$ if $\mathcal M$ and $\mathcal N^\omega$ are free with amalgamation over $\mathcal N$ inside $\mathcal M^\omega$. In other words, we have a canonical isomorphism $\mathcal M\vee \mathcal N^\omega\cong \mathcal M*_\mathcal N\mathcal N^\omega$.
\end{enumerate}    
\end{defn}

We next observe that these notions generalize the notions of free factors.

\begin{prop} The following hold:
\begin{enumerate}
    \item Let $G$ be a countable group. If a subgroup $H\leq G$ is a free factor of $G$, then it is an asymptotically free factor of $G$.

    \item Let $(\mathcal M,\tau)$ be a separable tracial von Neumann algebra. If a von Neumann subalgebra $\mathcal N\subset \mathcal M$ is a free factor of $\mathcal M$, then it is an asymptotically free factor of $\mathcal M$.
\end{enumerate}

\end{prop}

\begin{proof}
 (1) Assume that $H$ is a free factor of $G$, and let $K\leq G$ be a subgroup such that $G=H*K$. Then $G\vee H^\omega=K\vee H^\omega=K*H^\omega=(K*H)*_HH^\omega=G*_HH^\omega$.
 
 (2) Assume that $\mathcal N\subset \mathcal M$ is a free factor of $\mathcal M$, and let $\mathcal P\subset \mathcal N$ be a von Neumann subalgebra such that $\mathcal M=\mathcal N*\mathcal P$. Then $\mathcal M\vee \mathcal N^\omega=\mathcal P\vee\mathcal N^\omega=\mathcal P*\mathcal N^\omega=(\mathcal P*\mathcal N)*_\mathcal N\mathcal N^\omega=\mathcal M*_\mathcal N\mathcal N^\omega$. 
\end{proof}

The following result shows that the notions of asymptotically free factors for groups and von Neumann algebras are compatible.

\begin{thm}\label{equivalences}
    If $H$ is a subgroup of a countable group $G$, then the following  are equivalent:

    \begin{enumerate}
        \item For every $k\geq 1$ and $g_1,\dots,g_k\in G\setminus H$, we can find a finite set $L\subset H$ such that 
    for every  $h_1,\dots,h_k\in H\setminus L$, we have
$g_1h_1\cdots g_kh_k\not=1$.

        \item $\cL(H)$ is an asymptotically free factor of $\cL(G)$.
        
        \item $H$ is an asymptotically free factor of $G$.
    \end{enumerate}
\end{thm}

\begin{proof}
    (1) $\Rightarrow$ (2) Assume that (1) holds.
    By Lemma \ref{amalgam}, to prove (2), it suffices to show that if $k\geq 1$ and $x_1,\dots,x_k\in \cL(G)\ominus\cL(H)$, $y_1,\ldots,y_k\in\cL(H)^{\omega}\ominus\cL(H)$, then $\tau^\omega(x_1y_1\cdots x_ky_k)=0$. 
By using $\|\cdot\|_2$-approximations, we may assume that 
$x_i=u_{g_i}$, for some $g_i\in G\setminus H$, for every $1\leq i\leq k$. By (1), we can find  $L\subset H$ finite such that  $1\not\in g_1(H\setminus L)\cdots g_k(H\setminus L)$.

Denote by $e:\cL(H)\rightarrow\cL(H)$ the restriction to $\cL(H)$ of the orthogonal projection onto $\ell^2(L)$, i.e., $e(x)=\sum_{h\in L}\tau(xu_h^*)u_h$. Note that $\|e(x)\|\leq |L|\;\|x\|$, for every $x\in\cL(H)$.
For $1\leq i\leq k$, represent $y_i=(y_{i,n})_n$, where $\sup_n\|y_{i,n}\|<\infty$. 
Moreover, by Kaplansky's density theorem, we may assume that $y_{i,n}$ is finitely supported, for every $n\geq 1$.
Since $\lim_{n\rightarrow\omega}\tau(y_{i,n}u_h^*)=\tau^\omega(y_iu_h^*)=0$, for every $h\in H$, we get
$\lim_{n\rightarrow\omega}\|e(y_{i,n})\|=0$. Hence, $y_i=(z_{i,n})$ in $\cL(G)^\omega$, where $z_{i,n}=y_{i,n}-e(y_{i,n})$.  

 Since $z_{i,n}$ has finite support contained in $H\setminus L$, for every $1\leq i\leq k$, $u_{g_1}z_{1,n}\cdots u_{g_k}z_{k,n}$ is supported on $g_1(H\setminus L)\cdots g_k(H\setminus L)$.  Since 
$1\notin g_1(H\setminus L)\cdots g_k(H\setminus L)$, we conclude that $\tau(u_{g_1}z_{1,n}\cdots u_{g_k}z_{k,n})=0$, for every $n\geq 1$. This implies that 
 $\tau^\omega(x_1y_1\cdots x_ky_k)=\lim_{n\rightarrow\omega}\tau(u_{g_1}z_{1,n}\cdots u_{g_k}z_{k,n})=0.$

 (2) $\Rightarrow$ (3) This implication follows readily by using the natural inclusion $G^\omega\subset\mathcal U(\cL(G)^\omega)$.

 (3) $\Rightarrow$ (1) Assume that (3) holds, and suppose by contradiction that (1) is false. Enumerate $H=\{h_n\}_{n\geq 1}$ and define $L_n=\{h_1,\dots,h_n\}$, for every $n\geq 1$. Since (1) is false, we can find $g_1,\dots,g_k\in G\setminus H$, such that for every $n\geq 1$, there exist $h_{1,n},\cdots,h_{k,n}\in H\setminus L_n$ satisfying $g_1h_{1,n}\cdots g_kh_{k,n}=1$.
Therefore, putting $h_i=(h_{i,n})_n\in H^\omega$ we have $h_i\in H^\omega\setminus H$, for every $1\leq i\leq k$.
Since $g_1h_1\cdots g_kh_k=1$ and $g_1,\dots,g_k\in G\setminus H$, this contradicts the fact that $G$ and $H^\omega$ are free with amalgamation over $H$ inside $G^\omega$. 
\end{proof}

By combining Theorem \ref{equivalences} with Theorem \ref{evenlength}, we establish Theorem \ref{amalgfreeprodinLGw}.

\begin{thm} Let $G$ be a countable group with a hyperbolically embedded subgroup $H$. Then $\cL(H)$ is an asymptotically free factor of $\cL(G)$. 
\end{thm}

\begin{remark}

 Theorem \ref{equivalences} gives that  $G\vee H^{\omega}=G\ast_HH^{\omega}$, for $H$ hyperbolically embedded in $G$. However, the analogous decomposition does not hold  when considering the ultraproduct $C^*$-algebra $C_r^*(H)_{\omega}$. 
 Moreover, there is no conditional expectation $C_r^*(H)_{\omega}\to C_r^*(H)$ whenever $H$ is infinite. In fact, if $A$ is a unital separable $C^*$-algebra, then a conditional expectation $A_{\omega}\to A$ exists only if $A$ is finite dimensional. Indeed, for a nonprincipal ultrafilter $\omega$, suppose $\mathbb E: A_{\omega} \to A$ is a conditional expectation. Choose a MASA $C(K) \subset A$. If $(f_n) \in C(K, \mathbb R)$ is an increasing seuqnece, then $\mathbb E((f_n)_n)\in C(K)$ is easily seen to be a least upper bound for $(f_n)$. Hence, $C(K)$ is separable and sequentially monotone complete, which implies that $K$ is metrizable and extremally disconnected, and hence $K$ is a finite space. Since $C(K)$ is a finite dimensional MASA, $\text{dim}(A)\leq |K|^2<\infty$. 
    
\end{remark}

\section{Operator norm decay of averages in acylindrically hyperbolic groups}

Let $H<G$ be an infinite hyperbolically embedded subgroup.
In this section we show that elements $g \in G \setminus H$ exhibit averaging decay when conjugated by unitaries from a diffuse abelian subalgebra $\A \subset \mathcal L(H)$. This is a mixing phenomenon relative to the hyperbolically embedded subgroup $H$. 

\begin{thm}\label{normconv0}
Let $G$ be a countable group and $H<G$ an infinite hyperbolically embedded subgroup. Assume that $\A \subset \mathcal L(H)\subset \mathcal L(G)$ is a diffuse abelian von Neumann subalgebra.   
Let $g \in G \setminus H$. Then there exists a sequence of unitaries $(a_i)_{i \ge 1} \subset \mathscr U(\A)$ satisfying
\begin{equation}\label{zeroconv}
\lim_{n\rightarrow \infty }\left\| \frac{1}{n} \sum_{i=1}^n (a_i u_g a_i^*)\otimes (\bar a_i\bar u_g \bar a_i^*)\right\|=0, \end{equation}
where we consider the operator norm on $\mathfrak B(\ell^2 G\otimes \ell^2G)$.
\end{thm}

\begin{proof}
Set $K := \{g, g^{-1}\} \subset G \setminus H$.  Theorem \ref{nontrivial} provides a finite symmetric set $L \subset H$ such that
\begin{equation}\label{nontrivial altwords}
\text{$1\notin (K(H \setminus L))^{m}$, for every $m\geq 1$}.
\end{equation} 
By Corollary \ref{orthunitariesinM} we can find a sequence of unitaries $(a_i)_i\subset\mathscr U(\mathcal A)$ such that 
\begin{equation}\label{a_icondition}
    \tau(a_i^*a_ju_h)=0,\text{ for all }i\neq j\text{ and all }h\in L.
\end{equation}

To prove\eqref{normconv0}, denote $T_n:=\frac{1}{n} \sum_{i=1}^n (a_i u_g a_i^*)\otimes (\bar a_i\bar u_g \bar a_i^*) \in \cL(G)\overline\otimes \cL(G)^{\text{op}}$, and observe that by the spectral formula we have 
\begin{equation}\label{spformula}
\|T_n\|=\lim_{m\rightarrow \infty} \tau\big((T_n^* T_n)^m\big)^{\frac{1}{2m}}.
\end{equation}
From the definition of $T_n$, for every $m\geq 1$ we have
\begin{equation}\label{momentest1}
\tau\big((T_n^* T_n)^m\big) = \frac{1}{n^{2m}} \sum^n_{i_1,\dots,i_m, j_1,\dots,j_m=1} 
|\tau(a_{i_1} u_{g^{-1}} a_{i_1}^* a_{j_1} u_{g} a_{j_1}^* \cdots  a_{i_m} u_{g^{-1}} a_{i_m}^* a_{j_m} u_{g} a_{j_m}^*)|^2.
\end{equation}
In what follows, integers $i_1,\ldots,i_k,j_1,\ldots,j_k\in\{1,\ldots, n\}$ are said to satisfy the \emph{neighboring mismatch condition} if $i_1\neq j_1\neq i_2 \neq j_2 \neq \cdots\neq  i_k \neq j_k \neq i_1$.

\vspace{1mm}

\begin{claim}\label{freenes} For all integers $i_1 , \ldots, i_k, j_1, \ldots ,j_k \in \{1,\ldots,n\}$ satisfying the neighboring mismatch condition, we have \[\tau(a_{i_1} u_{g^{-1}} a_{i_1}^* a_{j_1} u_{g} a_{j_1}^* \cdots  a_{i_k} u_{g^{-1}} a_{i_k}^* a_{j_k} u_{g} a_{j_k}^*)=0. \]   
\end{claim}
\noindent\emph{Proof of Claim \ref{freenes}.} 
 The neighboring mismatch condition and  \eqref{a_icondition} give that $a_{i_1}^*a_{j_1}, a^*_{j_1} a_{i_2}, \ldots, a^*_{j_k}a_{i_1}$ are supported on  $H\setminus L$. Hence
$w=u_{g^{-1}} a_{i_1}^* a_{j_1} u_{g} a_{j_1}^* \cdots  a_{i_k} u_{g^{-1}} a_{i_k}^* a_{j_k} u_{g} a_{j_k}^*a_{i_1}$ is supported on $(K(H\setminus L))^{k}$. Since $1\notin (K(H\setminus L))^{k}$ by   \eqref{nontrivial altwords}, we derive that $\tau(w)=0$. $\hfill\blacksquare$

\vspace{1mm}

Next we will analyze how many terms in the sum \eqref{momentest1} are nonzero. Fix an arbitrary monomial $$w=  u_{g^{-1}} a_{i_1}^* a_{j_1} u_{g} a_{j_1}^* \cdots  a_{i_m} u_{g^{-1}} a_{i_m}^* a_{j_m} u_{g} a_{j_m}^*a_{i_1}.$$ If  the neighboring mismatch condition    $i_1\neq j_1\neq i_2 \neq j_2 \neq \cdots\neq  i_m \neq j_m \neq i_1$ holds, then by Claim \ref{freenes} we get $\tau(w)=0$. Otherwise, there is $1\leq k\leq m$ such that $i_k=j_k$ or $j_k=i_{k+1}$, where by convention we put $i_{m+1}=i_1$. 
Assume for simplicity that $i_k=j_k$, for some $1\leq k\leq m$.
Then after cancellation, we have
\begin{align*}
    w&= u_{g^{-1}} a_{i_1}^* a_{j_1} u_{g} a_{j_1}^* \cdots a_{j_{k-1}} u_g a^*_{j_{k-1}}a_{i_k}u_{g^{-1}}a^*_{i_k} a_{j_k} u_g a^*_{j_k} a_{i_{k+1}}u_{g^{-1}}a^*_{i_{k+1}}\cdots a_{i_m} u_{g^{-1}} a_{i_m}^* a_{j_m} u_{g} a_{j_m}^*a_{i_1}\\
    &=u_{g^{-1}} a_{i_1}^* a_{j_1} u_{g} a_{j_1}^* \cdots  a_{j_{k-1}} u_g a^*_{j_{k-1}}a_{i_{k+1}}u_{g^{-1}}a^*_{i_{k+1}}\cdots a_{i_m} u_{g^{-1}} a_{i_m}^* a_{j_m} u_{g} a_{j_m}^*a_{i_1}.
\end{align*}
Notice that we have cancelled from $w$ two blocks, where we call a block an expression of the form $a_iu_ga_i^*$ or $a_iu_{g^{-1}}a_i^*$. If the neighboring mismatch condition holds for the new indices, that is, if $i_1\neq j_1\neq i_2 \neq j_2 \neq \cdots \neq j_{k-1}\neq i_{k+1}\neq \cdots \neq  i_m \neq j_m \neq i_1$, Claim \ref{freenes} gives that $\tau(w)=0$. Otherwise, we will have a similar cancellation between two neighboring $i$ and $j$ indices. 
The process will stop after finitely many steps, at which point we will either have:
\begin{enumerate}
    \item $\tau(w)=0$, if along the process we encounter a neighboring mismatch condition, or 
    \item $w=1$ and hence $\tau(w)=1$, if there are cancellations between all $i$-indices and $j$-indices.
\end{enumerate}
Every complete cancellation described in part (2) corresponds to a non-crossing partition of the set $\{1,\ldots,2m\}$ placed on a circle. In this configuration, odd numbers represent $i$-indices and even numbers represent $j$-indices. The partition consists of $m$ disjoint pairs, each containing one odd and one even element, such that no two pairs cross. For an example, see the picture below. 

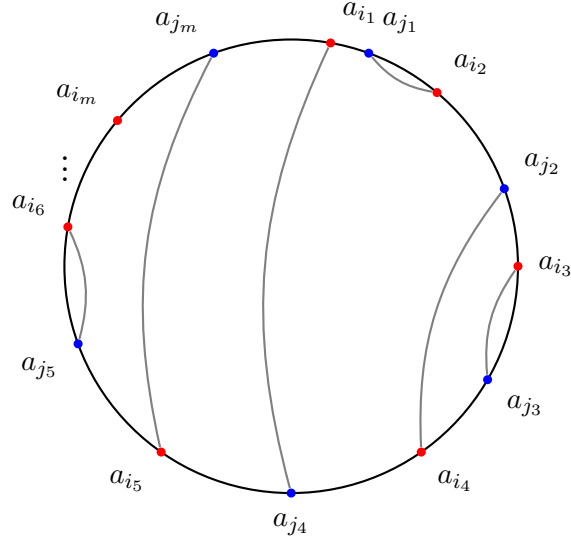
\begin{center}
\begin{tikzpicture}[
    dot/.style={circle, fill=#1, inner sep=1.2pt},
    label distance=2pt
]

    \draw[thick] (0,0) circle (3cm);

    \foreach \ang/\name/\col/\lab in {
        70/j1/blue/a_{j_1}, 
        50/i2/red/a_{i_2}, 
        20/j2/blue/a_{j_2}, 
        0/i3/red/a_{i_3}, 
        -30/j3/blue/a_{j_3}, 
        -55/i4/red/a_{i_4}, 
        -90/j4/blue/a_{j_4}, 
        -125/i5/red/a_{i_5}, 
        -160/j5/blue/a_{j_5}, 
        -190/i6/red/a_{i_6}, 
        -220/ik/red/a_{i_{m}}, 
        -250/jk/blue/a_{j_m}, 
        -280/i1/red/a_{i_1}%
    } {
        \node[dot=\col, label=\ang:$\lab$] (\name) at (\ang:3cm) {};
    }

    \node at (-205:3.3cm) {\Large $\vdots$};

    \begin{scope}[gray, thick, bend right=20]
        \draw (i1) to (j4);
        \draw (j1) to (i2);
        \draw (j2) to (i4);
        \draw (j5) to (i6);
        \draw (jk) to (i5);
    \end{scope}
    \begin{scope}[gray, thick, bend left=20]
    \draw (j3) to (i3);
    \end{scope}
\end{tikzpicture}
\end{center}


    
    
    







    
    
    



The number of such non-crossing partitions is given by the $m$-th Catalan number, $C_m=\frac{1}{m+1}\binom{2m}{m}$. 
Moreover, for every such complete canceling scheme, i.e. for each non-crossing partition, there are at most $n^m$ possible monomials $w$ in the sum \eqref{momentest1}. This arises because each of the $m$ connecting arcs (blocks of size two) allows for $n$ independent choices, resulting in total $n^m$ combinations for a given scheme. 
Therefore, using this counting method
and equation \eqref{momentest1}, we get that
\begin{equation*}
    \tau\big((T_n^* T_n)^m\big) \leq  \frac{1}{n^{2m}} \cdot \frac{1}{m+1} \binom{2m}{m} \cdot n^m = \frac{1}{n^{m}} \cdot \frac{1}{m+1} \binom{2m}{m}
\end{equation*}
and thus
\begin{align*}
    \tau\big((T_n^* T_n)^m\big)^{\frac{1}{2m}} \leq\frac{1}{\sqrt{n}} \cdot \left(\frac{1}{m+1} \binom{2m}{m}\right)^{\frac{1}{2m}}.
\end{align*}

Since $\lim_m \left(\frac{1}{m+1} \binom{2m}{m}\right)^{\frac{1}{2m}}=2$, together with \eqref{spformula} we get that $\|T_n\|\leq\frac{2}{\sqrt{n}}$, which proves \eqref{zeroconv}. \end{proof}

\begin{remark}
Denote by $(b_i)_{i\geq 1}$  the free generators of $\mathbb F_\infty$ and by $\lambda,\rho$ the left and right regular representations of $\mathbb F_\infty$.
With this notation, the proof of Theorem \ref{normconv0} shows   $\|T_n\|=\|S_n\|$, where $$\text{$S_n=\frac{1}{n}\sum_{i=1}^n\lambda(b_i)\otimes\lambda(b_i)^{\text{op}}\in\cL(\mathbb F_\infty)\overline{\otimes}\cL(\mathbb F_\infty)^{\text{op}}$.}$$ Since $\lambda\otimes\rho$ is a multiple of $\lambda$, we further have that  $\|S_n\|=\|s_n\|$, where  $s_n=\frac{1}{n}\sum_{i=1}^n\lambda(b_i)\in\cL(\mathbb F_\infty)$. 
It is a classical result of Akemann-Ostrand \cite{AO76} that $\|s_n\|=\frac{2\sqrt{n-1}}{n}$, for every $n\geq 2$.
\end{remark}

The  proof of Theorem \ref{normconv0} also yields the following result, whose details we leave to the reader.   
\begin{cor}\label{normconv01}
Let $G$ be a countable group and $H<G$ an infinite hyperbolically embedded subgroup. Assume that $\A \subset \mathcal L(H)\subset \mathcal L(G)$ is a diffuse abelian von Neumann subalgebra.   
Let $g \in G \setminus H$. Then there exists a sequence of unitaries $(a_i)_{i \ge 1} \subset \mathscr U(\A)$ satisfying
\begin{equation*}
\lim_{n\rightarrow \infty }\left\| \frac{1}{n} \sum_{i=1}^n a_i u_g a_i^* \right\|= 0, \end{equation*}
where we consider the operator norm on $\mathfrak B(\ell^2 G)$.
\end{cor}

\section{Proof of Theorem \ref{amenableabs1} and applications of amenable absorption}\label{applicationsam}

In this section we prove   Theorem \ref{amenableabs1} and discuss some applications.



\begin{thm}\label{amenableabs} Let $H<G$ be infinite groups with $H$  hyperbolically embedded in $G$. Then any amenable von Neumann subalgebra $\mathcal Q\subset \mathcal L(G)$ such that $\mathcal Q \cap  \mathcal L (H)$ is diffuse satisfies $\mathcal Q \subset \mathcal L(H)$.
\end{thm}

\begin{proof} 
We give two proofs. 
The first proof relies on an idea of Ozawa \cite{Oz15}. Put $\mathcal M=\mathcal L(G)$.
Since $\mathcal Q\subset\mathcal M$ is amenable, the canonical trace $\tau$ of $\mathcal M$ extends to a $\mathcal Q$-central state $\phi: \mathfrak B(L^2\mathcal M) \rightarrow \mathbb C$. 

Let $g\in G\setminus H$. Since $\mathcal Q\cap \mathcal L(H)$ is diffuse,  Theorem \ref{normconv0} implies that the norm closed convex hull of $\{au_ga^*\otimes \overline{a}\overline{u}_g\overline{a}^*\mid a\in\mathscr U(\mathcal Q\cap\mathcal L(H))\}$ contains $0$. Hence, the norm closed convex hull of $\{au_ga^*\otimes \overline{a}\overline{u}_g\overline{a}^*\mid a\in\mathscr U(\mathcal Q)\}$ also contains $0$.
 Since $\mathcal Q$ is amenable, \cite[Lemma]{Oz15} implies that $\varphi(u_g^*\mathcal Q)=\{0\}$.
 Since this holds for every $g\in G\setminus H$, we conclude that $\mathcal Q\subset\mathcal L(H)$.

 Our second proof derives the conclusion directly, without appealing to states. To prove the conclusion, it suffices to show that $\mathbb{E}_{\mathcal Q}(u_g)=0$, for every $g\in G\setminus H$. 
Fix $g\in G\setminus H$.
By Theorem \ref{normconv0}, we can find unitaries $(a_i)_{i\geq 1}\subset\mathscr U(\mathcal Q\cap\mathcal L(H))$ such that denoting
 $u_i=a_iu_ga_i^*$, for $i\geq 1$, we have $\|T_n\|\rightarrow 0$, where $T_n=\frac{1}{n}\sum_{i=1}^nu_i\otimes u_i^{\text{op}}\in \cL(G)\overline{\otimes}\cL(G)^{\text{op}}$. Since $\mathbb {E}_{\mathcal Q\overline{\otimes}\mathcal Q^{\text{op}}}$ is norm decreasing we get 
\begin{equation}\label{Tn}\left\|\frac{1}{n}\sum_{i=1}^n\mathbb{E}_\mathcal Q(u_i)\otimes\mathbb{E}_\mathcal Q(u_i)^{\text{op}}\right\|=\|\mathbb{E}_{\mathcal Q\overline{\otimes}\mathcal Q^{\text{op}}}(T_n)\|\rightarrow 0.\end{equation}
On the other hand, since $\mathcal Q$ is amenable we have that $|\tau(\sum_{i=1}^nx_iy_i^*)|\leq \|\sum_{i=1}^nx_i\otimes y_i^{\text{op}}\|$, for every $x_1,y_1,\dots,x_n,y_n\in \mathcal Q$.
Thus, we derive that
$$\left\|\frac{1}{n}\sum_{i=1}^n\mathbb{E}_\mathcal Q(u_i)\otimes\mathbb{E}_\mathcal Q(u_i)^{\text{op}}\right\|\geq \left|\frac{1}{n}\sum_{i=1}^n\tau(\mathbb{E}_\mathcal Q(u_i)\mathbb{E}_\mathcal Q(u_i)^*)\right|=\frac{1}{n}\sum_{i=1}^n\|\mathbb{E}_\mathcal Q(u_i)\|_2^2=\|\mathbb{E}_\mathcal Q(u_g)\|_2^2.$$
In combination with \eqref{Tn} we conclude that $\mathbb E_{\mathcal Q}(u_g)=0$, as desired.
\end{proof}

In the remaining part of this section we present several applications of our amenable absorption theorem. 


\begin{cor}\label{cortomain} Let $G$ be a group with a hyperbolically embedded infinite subgroup $H$.
Assume that $H$ is center-by-finite, i.e., $H/Z(H)$ is a finite group.
Then for any type II$_1$ amenable von Neumann subalgebra $Q\subset \mathcal L(G)$, the intersection $\mathcal Q \cap \mathcal L(H)$ is completely atomic.    
\end{cor}

\begin{cor}\label{absorptioncorner} Let $G$ be a group with a hyperbolically embedded infinite subgroup $H$. Let $p\in \mathcal L(H)$ be a nonzero projection. Then  any amenable von Neumann subalgebra $\mathcal{Q}\subset p\cL(G)p$ such that  $\mathcal{Q}\cap p \cL(H)p $ is diffuse satisfies $\mathcal{Q}\subset p\cL(H) p$.
\end{cor}

\begin{proof} Let $\mathcal A \subset (1-p)\cL(H)(1-p)$ be any diffuse amenable subalgebra. For instance, take $\mathcal A$ to be a MASA of $(1-p)\cL(H)(1-p)$. Then $\mathcal P := \mathcal Q \oplus \mathcal A$ is an amenable von Neumann subalgebra of $\cL(G)$ with $\mathcal P \cap \mathcal L(H)$ diffuse. By Theorem \ref{amenableabs} it follows that $\mathcal P\subset \cL(H)$, and thus $\mathcal{Q}\subset p\cL(H) p$.\end{proof}

\begin{prop} Let $\mathcal N \subset \mathcal M$ be von Neumann algebras such that for every amenable von Neumann subalgebra $\mathcal A \subset \mathcal M$ that intersects diffusely with $\mathcal N$ must be contained in $\mathcal N$. Let $\mathcal G\leqslant \mathscr U(\mathcal N)$ be any amenable group of unitaries such that $\mathcal G''$ is diffuse. Then for any amenable von Neumann algebra $\mathcal B$ such that $\mathcal G\subset \mathscr N _{\mathcal M}(\mathcal B)$ we must have that $(\mathcal B\mathcal G)''\subset \mathcal N$.
\end{prop}

\begin{proof} Since $\mathcal B$ and $\mathcal G$ are amenable it follows that $(\mathcal B\mathcal G)''\subset \mathcal M$ is an amenable von Neumann subalgebra. Moreover, since the intersection $(\mathcal B \mathcal G )''\cap\mathcal N$ contains the diffuse von Neumann algebra $\mathcal G''$ as a unital subalgebra, the hypothesis assumption implies that $(\mathcal B \mathcal G )''\subseteq \mathcal N$.\end{proof}

When combined with Theorem \ref{amenableabs}, the last results implies the following corollary.


\begin{cor}\label{absorptionnormalizer} Let $G$ be a group with a hyperbolically embedded infinite subgroup $H$. Let $\mathcal{G}\leqslant\mathscr{U}(\cL(H))$ be any amenable group of unitaries with $\mathcal{G}''$  diffuse. Then any amenable von Neumann subalgebra $\mathcal{B}$ of $\mathcal L(G)$ such that $\mathcal{G}\subset\mathscr{N}_{\cL(G)}(\mathcal{B})$ satisfies $(\mathcal{BG})''\subset\cL(H)$.
\end{cor}

\subsection{Peterson--Thom conjecture}\label{PetersonThom}

A recurring theme in the study of maximal amenability is the extent to which amenable subalgebras are constrained by the ambient nonamenable von Neumann algebra. One particularly natural question is whether a diffuse amenable subalgebra determines a unique maximal amenable extension. Motivated by striking parallels between indecomposability and malnormality phenomena for groups with positive $L^2$-Betti number and rigidity phenomena in free group factors, Peterson and Thom formulated in \cite{PT11} the following deep conjecture. Their work built on earlier results of Ozawa, Popa, Peterson, and Jung \cite{OP07,Pe09,Jun07}, as well as on their own structural results for groups with positive $L^2$-Betti number \cite{PT11}.

\begin{conj} Fix $r>1$. If $\mathcal{Q}$ is a diffuse amenable von Neumann subalgebra of $\cL(\mathbb{F}_r)$, then there exists a unique maximal amenable von Neumann subalgebra $\mathcal{P}$ of $\cL(\mathbb{F}_r)$ with $\mathcal{Q}\subset \mathcal{P}$.
\end{conj}

We note that a positive answer to this conjecture would automatically recover several fundamental structural results for free group factors, including primeness \cite{Ge96}, the absence of Cartan subalgebras \cite{Voi96}, and, more generally, strong solidity \cite{OP07}. Moreover, it would yield substantial strengthenings of these properties; see, for instance, \cite{HJKE25}.

Several results on amenable absorption (e.g., \cite{HU15, Ho14, BH16}) confirmed the Peterson-Thom conjecture in specific cases and were widely seen as strong supporting evidence. This is because the conjecture can be reformulated in terms of amenable absorption: for every $r > 1$, any maximal amenable subalgebra $\:\mathcal{N} \subset \mathcal{L}(\mathbb{F}_r)$ should have the amenable absorbing property. Later, in \cite[Theorem 1.1]{Hay22}, Hayes introduced a new approach to proving the Peterson-Thom conjecture using Voiculescu’s free entropy dimension theory \cite{Voi94, Voi96} and techniques from random matrix theory, which required the use of the 1-bounded entropy \cite{Hay18}. The conjecture was ultimately proven true by Belinschi and Capitaine \cite{BC22}, and independently by Bordevane and Collins \cite{BC23}, who relied on the equivalent formulation of the conjecture within the framework of random matrix theory developed in \cite{Hay22}. Other proofs of the conjecture include \cite{CGvH,Parr}.

In our case, we interpret our results on amenable absorption for groups $G$ with hyperbolically embedded subgroups $H$ as supporting evidence for the following conjecture:

\begin{conj}\label{PetersonThomhyperbolic} (Peterson-Thom conjecture for hyperbolic groups) Let $G$ be any hyperbolic group. Then any two diffuse amenable von Neumann subalgebras $\mathcal A,\:\mathcal B\subset  \cL(G)$ such that $\mathcal A \cap\mathcal B$ is diffuse, generate a von Neumann subalgebra $\mathcal A \vee\mathcal B\subset \mathcal L(G)$ that is still amenable.
\end{conj}

Theorem \ref{amenableabs} confirms Conjecture \ref{PetersonThomhyperbolic} in several nontrivial cases; for example, when $\mathcal{A} \subset \mathcal{L}(G)$ is an arbitrary diffuse amenable subalgebra and $\mathcal{B} = \mathcal{L}(H)$ is any diffuse amenable subalgebra arising from a subgroup $H < G$. Indeed, since $\mathcal{B}$ is diffuse and amenable, $H$ must be infinite and amenable. In particular, $H$ is contained in an infinite maximal amenable subgroup $K < G$. By \cite{Osi06b}, such a group $K$ is hyperbolically embedded in $G$ and it is elementary, and thus, $K$ is a center-by-finite subgroup of $G$. Now, since $\mathcal{A} \cap \mathcal{B}$ is diffuse, we also have that $\mathcal{A} \cap \mathcal{L}(K)$ is diffuse. Applying Theorem \ref{amenableabs}, we conclude that $\mathcal{A} \subset \mathcal{L}(K)$. Since $\mathcal{B}$ is already contained in $\mathcal{L}(K)$, it follows that $\mathcal{A} \vee \mathcal{B} \subset \mathcal{L}(K)$. Finally, as $K$ is amenable, so is the von Neumann subalgebra $\mathcal{A} \vee \mathcal{B}$.

\end{document}